\documentclass{amsart}
\usepackage{amssymb}
\usepackage{amsmath}
\usepackage{amsfonts}

\newtheorem{theorem}{Theorem}
\theoremstyle{plain}

\newtheorem{corollary}{Corollary}

\newtheorem{definition}{Definition}

\newtheorem{proposition}{Proposition}

\numberwithin{equation}{section}
\input{tcilatex}

\begin{document}
\title[On a Robin problem involving weighted]{On some general multiplying
solutions results of a Robin problem }
\author{Ismail Ayd\i n}
\address{Sinop University\\
Faculty of Arts and Sciences\\
Department of Mathematics\\
Sinop, Turkey}
\email{iaydin@sinop.edu.tr}
\author{Cihan Unal}
\address{Assessment, Selection and Placement Center\\
Ankara, Turkey}
\email{cihanunal88@gmail.com}
\thanks{}
\dedicatory{}
\thanks{}

\begin{abstract}
By applying Ricceri's variational principle, we demonstrate the existence of
solutions for the following Robin problem%
\begin{equation*}
\left\{ 
\begin{array}{cc}
-\func{div}\left( \omega _{1}(x)\left\vert \nabla u\right\vert
^{p(x)-2}\nabla u\right) =\lambda \omega _{2}(x)f(x,u), & x\in \Omega \\ 
\omega _{1}(x)\left\vert \nabla u\right\vert ^{p(x)-2}\frac{\partial u}{%
\partial \upsilon }+\beta (x)\left\vert u\right\vert ^{p(x)-2}u=0, & x\in
\partial \Omega ,%
\end{array}%
\right.
\end{equation*}%
in $W_{\omega _{1},\omega _{2}}^{1,p(.)}\left( \Omega \right) $ under some
appropriate conditions.
\end{abstract}

\maketitle

\section{Introduction}

Let $\Omega \subset 
%TCIMACRO{\U{211d} }%
%BeginExpansion
\mathbb{R}
%EndExpansion
^{N}$ $(N\geq 2)$ be a bounded smooth domain. Assume that $\omega _{1}$ and $%
\omega _{2}$ are weight functions. The aim of this study is to discuss the
three solutions for the following Robin problem 
\begin{equation}
\left\{ 
\begin{array}{cc}
-\func{div}\left( \omega _{1}(x)\left\vert \nabla u\right\vert
^{p(x)-2}\nabla u\right) =\lambda \omega _{2}(x)f(x,u), & x\in \Omega \\ 
\omega _{1}(x)\left\vert \nabla u\right\vert ^{p(x)-2}\frac{\partial u}{%
\partial \upsilon }+\beta (x)\left\vert u\right\vert ^{p(x)-2}u=0, & x\in
\partial \Omega ,%
\end{array}%
\right.  \label{P}
\end{equation}%
where $\frac{\partial u}{\partial \upsilon }$ is the outer unit normal
derivative of $u$ with respect to $\partial \Omega $, $\lambda >0$, $p,q\in
C\left( \overline{\Omega }\right) $ with $\underset{x\in \overline{\Omega }}{%
\inf }p(x)>1$, and $\beta \in L^{\infty }\left( \partial \Omega \right) $
such that $\beta ^{-}=\underset{x\in \partial \Omega }{\inf }\beta (x)>0$.

In recent years, the investigating of the existence of weak solutions of
partial differential equations involving weighted $p(x)$-Laplacian in
variable exponent (weighted or unweighted) Sobolev spaces has been very
popular (see \cite{ay}, \cite{aydn}, \cite{den}, \cite{fan}, \cite{FanZ}, 
\cite{mi}, \cite{UnalAy2}). Because some such type of equations can explain
several physical problems such as electrorheological fluids,\ image
processing, elastic mechanics, fluid dynamics and calculus of variations,
see \cite{Ha}, \cite{mih}, \cite{Ru}, \cite{Zh}.

The Robin problem involving $p(x)$-Laplacian was studied by several authors,
see \cite{al}, \cite{Ch}, \cite{deng}, \cite{ke}, \cite{Tso}. In 2013,
Tsouli et al. \cite{Ts} obtained some results about weak solutions of the
following Robin problem%
\begin{equation}
\begin{array}{cc}
-\func{div}\left( \left\vert \nabla u\right\vert ^{p(x)-2}\nabla u\right)
=\lambda f(x,u), & x\in \Omega \\ 
\left\vert \nabla u\right\vert ^{p(x)-2}\frac{\partial u}{\partial \upsilon }%
+\beta (x)\left\vert u\right\vert ^{p(x)-2}u=0, & x\in \partial \Omega%
\end{array}
\label{P1}
\end{equation}%
using the variational methods under some suitable conditions for the
function $f$. In addition, they showed that the problem (\ref{P1}) has at
least three solutions.

In the light of the articles mentioned above, we discuss the existence of
multiplicity solutions of the problem (\ref{P}) in the variable exponent
Sobolev spaces $W_{\omega _{1},\omega _{2}}^{1,p(.)}\left( \Omega \right) $
with respect to two different weight functions $\omega _{1}$ and $\omega
_{2} $. Moreover, we introduce a more general norm in compared to the norm
given by Deng \cite{deng}. Finally, we find more general results than \cite%
{Ts} using the technical approach, which is mainly based on Ricceri's
theorem.

\section{Notation and preliminaries}

Let $\Omega $ be a bounded open subset of $%
%TCIMACRO{\U{211d} }%
%BeginExpansion
\mathbb{R}
%EndExpansion
^{N}$ with a smooth boundary $\partial \Omega $. Then, the set is defined by%
\begin{equation*}
C_{+}\left( \overline{\Omega }\right) =\left\{ p\in C\left( \overline{\Omega 
}\right) :\inf_{x\in \overline{\Omega }}p(x)>1\right\} ,
\end{equation*}%
where $C\left( \overline{\Omega }\right) $ consists of all continuous
functions on $\overline{\Omega }$. For any $p\in C_{+}\left( \overline{%
\Omega }\right) $, we indicate%
\begin{equation*}
p^{-}=\inf_{x\in \Omega }p(x)\text{ and }p^{+}=\sup_{x\in \Omega }p(x).
\end{equation*}%
Let $p\in C_{+}\left( \overline{\Omega }\right) $ and $1<p^{-}\leq p(.)\leq
p^{+}<\infty .$ The space $L^{p(.)}(\Omega )$ is defined by%
\begin{equation*}
L^{p(.)}(\Omega )=\left\{ u\left\vert u:\Omega \rightarrow 
%TCIMACRO{\U{211d} }%
%BeginExpansion
\mathbb{R}
%EndExpansion
\text{ is measurable and }\int\limits_{\Omega }\left\vert u(x)\right\vert
^{p(x)}dx<\infty \right. \right\}
\end{equation*}%
with the (Luxemburg) norm%
\begin{equation*}
\left\Vert u\right\Vert _{p(.)}=\inf \left\{ \lambda >0:\varrho
_{p(.)}\left( \frac{u}{\lambda }\right) \leq 1\right\} \text{,}
\end{equation*}%
where 
\begin{equation*}
\varrho _{p(.)}(u)=\int\limits_{\Omega }\left\vert u(x)\right\vert ^{p(x)}dx
\end{equation*}%
,see \cite{ko}. If $\left\Vert u\right\Vert _{p(.),\omega }=\left\Vert
u\omega ^{\frac{1}{p(.)}}\right\Vert _{p(.)}<\infty $, then $u\in L_{\omega
}^{p(.)}(\Omega )$ where $\omega $ is a weight function from $\Omega $ to $%
\left( 0,\infty \right) $. It is known that the space $\left( L_{\omega
}^{p(.)}(\Omega ),\left\Vert .\right\Vert _{p(.),\omega }\right) $ is a
Banach space. The dual space of $L_{\omega }^{p(.)}(\Omega )$ is $L_{\omega
^{\ast }}^{r(.)}(\Omega )$ where $\frac{1}{p(.)}+\frac{1}{r(.)}=1$ and $%
\omega ^{\ast }=\omega ^{1-r\left( .\right) }=\omega ^{-\frac{1}{p(.)-1}}.$
If $\omega \in L^{\infty }\left( \Omega \right) $, then $L_{\omega
}^{p(.)}=L^{p(.)}$, see \cite{ayd}, \cite{ayn}.

\begin{proposition}
(see \cite{ayd}, \cite{FanZ}) For all $u,v\in L_{\omega }^{p(.)}\left(
\Omega \right) $, we have

\begin{enumerate}
\item[\textit{(i)}] $\left\Vert u\right\Vert _{p(.),\omega }<1$ (resp.$=1,>1$%
) if and only if $\varrho _{p(.),\omega }(u)<1$ (resp.$=1,>1$),

\item[\textit{(ii)}] $\left\Vert u\right\Vert _{p(.),\omega }^{p^{-}}\leq
\varrho _{p(.),\omega }(u)\leq \left\Vert u\right\Vert _{p(.),\omega
}^{p^{+}}$ with $\left\Vert u\right\Vert _{p(.),\omega }>1,$

\item[\textit{(iii)}] $\left\Vert u\right\Vert _{p(.),\omega }^{p^{+}}\leq
\varrho _{p(.),\omega }(u)\leq \left\Vert u\right\Vert _{p(.),\omega
}^{p^{-}}$ with $\left\Vert u\right\Vert _{p(.),\omega }<1,$

\item[\textit{(iv)}] $\min \left\{ \left\Vert u\right\Vert _{p(.),\omega
}^{p^{-}},\left\Vert u\right\Vert _{p(.),\omega }^{p^{+}}\right\} \leq
\varrho _{p(.),\omega }(u)\leq \max \left\{ \left\Vert u\right\Vert
_{p(.),\omega }^{p^{-}},\left\Vert u\right\Vert _{p(.),\omega
}^{p^{+}}\right\} $,

\item[\textit{(v)}] $\min \left\{ \varrho _{p(.),\omega }(u)^{\frac{1}{p^{-}}%
},\varrho _{p(.),\omega }(u)^{\frac{1}{p^{+}}}\right\} \leq \left\Vert
u\right\Vert _{p(.),\omega }\leq \max \left\{ \varrho _{p(.),\omega }(u)^{%
\frac{1}{p^{-}}},\varrho _{p(.),\omega }(u)^{\frac{1}{p^{+}}}\right\} $,

\item[\textit{(vi)}] $\varrho _{p(.),\omega }(u-v)\rightarrow 0$ if and only
if $\left\Vert u-v\right\Vert _{p(.),\omega }\rightarrow 0$
\end{enumerate}

where $\varrho _{p(.),\omega }(u)$ is defined by the integral $%
\int\limits_{\Omega }\left\vert u(x)\right\vert ^{p(x)}\omega (x)dx.$
\end{proposition}

\begin{definition}
Let $\omega ^{-\frac{1}{p(.)-1}}\in L_{loc}^{1}\left( \Omega \right) $. The
space $W_{\omega }^{k,p(.)}\left( \Omega \right) $ is defined by%
\begin{equation*}
W_{\omega }^{k,p(.)}\left( \Omega \right) =\left\{ u\in L_{\omega
}^{p(.)}(\Omega ):D^{\alpha }u\in L_{\omega }^{p(.)}(\Omega ),0\leq
\left\vert \alpha \right\vert \leq k\right\}
\end{equation*}%
equipped with the norm 
\begin{equation*}
\left\Vert u\right\Vert _{\omega }^{k,p(.)}=\sum\limits_{0\leq \left\vert
\alpha \right\vert \leq k}\left\Vert D^{\alpha }u\right\Vert _{p(.),\omega }
\end{equation*}%
where $\alpha \in 
%TCIMACRO{\U{2115} }%
%BeginExpansion
\mathbb{N}
%EndExpansion
_{0}^{N}$ is a multi-index, $\left\vert \alpha \right\vert =\alpha
_{1}+\alpha _{2}+...+\alpha _{N}$ and $D^{\alpha }=\frac{\partial
^{\left\vert \alpha \right\vert }}{\partial _{x_{1}}^{\alpha
_{1}}...\partial _{x_{N}}^{\alpha _{N}}}$. In particular, the space $%
W_{\omega }^{1,p(.)}\left( \Omega \right) $ is defined by 
\begin{equation*}
W_{\omega }^{1,p(.)}\left( \Omega \right) =\left\{ u\in L_{\omega
}^{p(.)}(\Omega ):\left\vert \nabla u\right\vert \in L_{\omega
}^{p(.)}(\Omega )\right\}
\end{equation*}%
equipped with the norm 
\begin{equation*}
\left\Vert u\right\Vert _{\omega }^{1,p(.)}=\left\Vert u\right\Vert
_{p(.),\omega }+\left\Vert \nabla u\right\Vert _{p(.),\omega }.
\end{equation*}%
The space $W_{\omega ^{\ast }}^{-1,r(.)}\left( \Omega \right) $ is the
topological dual for $W_{\omega }^{1,p(.)}\left( \Omega \right) $ where $%
\frac{1}{p(.)}+\frac{1}{r(.)}=1$ and $\omega ^{\ast }=\omega ^{1-r\left(
.\right) }=\omega ^{-\frac{1}{p(.)-1}}$. Moreover, the space $W_{\omega
}^{1,p(.)}\left( \Omega \right) $ is a separable and reflexive Banach space,
see \cite{ayn}.
\end{definition}

Let $\omega _{1}^{-\frac{1}{p(.)-1}},\omega _{2}^{-\frac{1}{p(.)-1}}\in
L_{loc}^{1}\left( \Omega \right) $. The space $W_{\omega _{1},\omega
_{2}}^{1,p(.)}\left( \Omega \right) $ is defined by%
\begin{equation*}
W_{\omega _{1},\omega _{2}}^{1,p(.)}\left( \Omega \right) =\left\{ u\in
L_{\omega _{2}}^{p(.)}(\Omega ):\left\vert \nabla u\right\vert \in L_{\omega
_{1}}^{p(.)}(\Omega )\right\}
\end{equation*}%
equipped with the norm%
\begin{equation*}
\left\Vert u\right\Vert _{\omega _{1},\omega _{2}}^{1,p(.)}=\left\Vert
\nabla u\right\Vert _{p(.),\omega _{1}}+\left\Vert u\right\Vert
_{p(.),\omega _{2}}.
\end{equation*}

Let $d\sigma $ be the measure on $\partial \Omega $. We can define the space 
$L_{\omega }^{p(.)}(\partial \Omega )$ similarly by%
\begin{equation*}
L_{\omega }^{p(.)}(\partial \Omega )=\left\{ u\left\vert u:\partial \Omega
\longrightarrow 
%TCIMACRO{\U{211d} }%
%BeginExpansion
\mathbb{R}
%EndExpansion
\text{ measurable and }\int\limits_{\partial \Omega }\left\vert
u(x)\right\vert ^{p(x)}\omega (x)d\sigma <+\infty \right. \right\}
\end{equation*}%
with the Luxemburg norm $\left\Vert .\right\Vert _{p(.),\omega ,\partial
\Omega }$. Then $\left( L_{\omega }^{p(.)}(\partial \Omega ),\left\Vert
.\right\Vert _{p(.),\omega ,\partial \Omega }\right) $ is a Banach space. If 
$\omega \in L^{\infty }\left( \Omega \right) $, then $L_{\omega
}^{p(.)}=L^{p(.)}.$

\begin{theorem}
\label{teo1}(see \cite{aydn})Let $\omega _{1}^{-\alpha \left( .\right) }\in
L^{1}\left( \Omega \right) $ with $\alpha \left( x\right) \in \left( \frac{N%
}{p(x)},\infty \right) \cap \left[ \frac{1}{p(x)-1},\infty \right) $. If we
define the variable exponent $p_{\ast }(x)=\frac{\alpha \left( x\right) p(x)%
}{\alpha \left( x\right) +1}$ with $N<p_{\ast }^{-}$, then we have the
embeddings $W_{\omega _{1},\omega _{2}}^{1,p(.)}\left( \Omega \right)
\hookrightarrow W^{1,p_{\ast }(.)}\left( \Omega \right) $ and $W_{\omega
_{1},\omega _{2}}^{1,p(.)}\left( \Omega \right) \hookrightarrow
\hookrightarrow C\left( \overline{\Omega }\right) .$
\end{theorem}

\begin{corollary}
\label{cor1}Since $W_{\omega _{1},\omega _{2}}^{1,p(.)}\left( \Omega \right)
\hookrightarrow \hookrightarrow C\left( \overline{\Omega }\right) $, then
there exists a $c_{1}>0$ such that 
\begin{equation*}
\left\Vert u\right\Vert _{\infty }\leq c_{1}\left\Vert u\right\Vert _{\omega
_{1},\omega _{2}}^{1,p(.)}
\end{equation*}%
for any $u\in W_{\omega _{1},\omega _{2}}^{1,p(.)}\left( \Omega \right) $
where $\left\Vert u\right\Vert _{\infty }=\sup_{x\in \overline{\Omega }}u(x)$
for $u\in C\left( \overline{\Omega }\right) $.
\end{corollary}

For $A\subset \overline{\Omega }$, denote by $p^{-}(A)=\underset{x\in A}{%
\text{inf}}p(x)$ and $p^{+}(A)=\underset{x\in A}{\text{sup}}p(x)$. We define%
\begin{equation*}
p^{\partial }\left( x\right) =\left( p(x)\right) ^{\partial }=\left\{ 
\begin{array}{cc}
\frac{\left( N-1\right) p(x)}{N-p(x)}, & \text{if }p\left( x\right) <N, \\ 
+\infty , & \text{if }p\left( x\right) \geq N,%
\end{array}%
\right.
\end{equation*}%
and%
\begin{equation*}
p_{r(x)}^{\partial }\left( x\right) =\frac{r(x)-1}{r(x)}p^{\partial }\left(
x\right)
\end{equation*}%
for any $x\in \partial \Omega $, where $r\in C\left( \partial \Omega \right) 
$ with $r^{-}=\underset{x\in \partial \Omega }{\text{inf}}r(x)>1.$

\begin{theorem}
\label{teo2}(see \cite{de})Assume that the boundary of $\Omega $ possesses
the cone property and $p\in C\left( \overline{\Omega }\right) $ with $%
p^{-}>1 $. Suppose that $\omega _{1}\in L^{r(.)}(\partial \Omega )$, $r\in
C\left( \partial \Omega \right) $ with $r(x)>\frac{p^{\partial }\left(
x\right) }{p^{\partial }\left( x\right) -1}$ for all $x\in \partial \Omega $%
. If $q\in C\left( \partial \Omega \right) $ and $1\leq
q(x)<p_{r(x)}^{\partial }\left( x\right) $ for all $x\in \partial \Omega $,
then there is a compact embedding $W^{1,p(.)}\left( \Omega \right)
\hookrightarrow L_{\omega _{1}}^{q(.)}(\partial \Omega )$. In particular,
there is a compact embedding $W^{1,p(.)}\left( \Omega \right)
\hookrightarrow L^{q(.)}(\partial \Omega )$, where $1\leq q(x)<p^{\partial
}\left( x\right) $ for all $x\in \partial \Omega $.
\end{theorem}

\begin{corollary}
(see \cite{de})

\begin{enumerate}
\item[\textit{(i)}] There is a compact embedding $W^{1,p(.)}\left( \Omega
\right) \hookrightarrow L^{p(.)}(\partial \Omega )$, where $1\leq
p(x)<p^{\partial }\left( x\right) $ for all $x\in \partial \Omega $.

\item[\textit{(ii)}] There is a compact embedding $W^{1,p(.)}\left( \Omega
\right) \hookrightarrow L_{\omega }^{p(.)}(\partial \Omega )$, where $1\leq
p(x)<p_{r(x)}^{\partial }\left( x\right) <p^{\partial }\left( x\right) $ for
all $x\in \partial \Omega $.
\end{enumerate}
\end{corollary}

\begin{corollary}
If $p(x)<p_{\ast ,r(x)}^{\partial }\left( x\right) <p_{\ast }^{\partial
}\left( x\right) $, $\forall x\in \partial \Omega $, then we have the
compact embeddings $W_{\omega _{1},\omega _{2}}^{1,p(.)}\left( \Omega
\right) \hookrightarrow W^{1,p_{\ast }(.)}\left( \Omega \right)
\hookrightarrow L_{\omega _{1}}^{p(.)}(\partial \Omega )$ by Theorem \ref%
{teo1} and Theorem \ref{teo2}.
\end{corollary}

\begin{theorem}
\label{teo3}(see \cite{deng})Assume that the boundary of $\Omega $ possesses
the cone property and $p\in C\left( \overline{\Omega }\right) $ with $%
p^{-}>1 $. If $q\in C\left( \overline{\Omega }\right) $ and $1\leq
q(x)<p^{\ast }(x)$ for all $x\in \overline{\Omega }$, then there is a
compact embedding $W^{1,p(.)}\left( \Omega \right) \hookrightarrow
L^{q(.)}(\Omega ),$ where%
\begin{equation*}
p^{\ast }(x)=\left\{ 
\begin{array}{cc}
\frac{Np(x)}{N-p(x)}, & \text{if }p\left( x\right) <N \\ 
+\infty , & \text{if }p\left( x\right) \geq N.%
\end{array}%
\right.
\end{equation*}
\end{theorem}

\begin{corollary}
\label{cor4}Let $N<p_{\ast }^{-}$ and $1\leq q(x)<\left( p_{\ast }\right)
^{\ast }(x)$ for all $x\in \overline{\Omega }$, then we have the compact
embedding $W_{\omega _{1},\omega _{2}}^{1,p(.)}\left( \Omega \right)
\hookrightarrow L^{q(.)}(\Omega ).$
\end{corollary}

\begin{proof}
By Theorem \ref{teo1} and Theorem \ref{teo3}, we have the continuous
embedding $W_{\omega _{1},\omega _{2}}^{1,p(.)}\left( \Omega \right)
\hookrightarrow W^{1,p_{\ast }(.)}\left( \Omega \right) $ and the compact
embedding $W^{1,p_{\ast }(.)}\left( \Omega \right) \hookrightarrow
L^{q(.)}(\Omega ).$ Thus it is easy to see that the compact embedding $%
W_{\omega _{1},\omega _{2}}^{1,p(.)}\left( \Omega \right) \hookrightarrow
L^{q(.)}(\Omega )$ is valid.
\end{proof}

If we apply the technique in \cite[Theorem 2.1]{deng}, then we prove the
following theorem similarly. Moreover, due to this theorem we can find out
the existence of weak solutions of the problem (\ref{P}).

\begin{theorem}
Let $\beta \in L^{\infty }\left( \partial \Omega \right) $ such that $\beta
^{-}=\underset{x\in \partial \Omega }{\inf }\beta (x)>0$. Then, the norm $%
\left\Vert u\right\Vert _{\beta (x)}$ is defined by%
\begin{equation*}
\left\Vert u\right\Vert _{\beta (x)}=\inf \left\{ \tau
>0:\int\limits_{\Omega }\omega _{1}(x)\left\vert \frac{\nabla u(x)}{\tau }%
\right\vert ^{p(x)}dx+\int\limits_{\partial \Omega }\beta (x)\left\vert 
\frac{u(x)}{\tau }\right\vert ^{p(x)}d\sigma \leq 1\right\}
\end{equation*}%
for any $u\in W_{\omega _{1},\omega _{2}}^{1,p(.)}\left( \Omega \right) $.
Moreover, $\left\Vert .\right\Vert _{\beta (x)}$ and $\left\Vert
.\right\Vert _{\omega _{1},\omega _{2}}^{1,p(.)}$ are equivalent on $%
W_{\omega _{1},\omega _{2}}^{1,p(.)}\left( \Omega \right) $.
\end{theorem}

\begin{proposition}
\label{pro2}(see \cite{deng}) Let $I_{\beta (x)}(u)=\int\limits_{\Omega
}\omega _{1}(x)\left\vert \nabla u(x)\right\vert
^{p(x)}dx+\int\limits_{\partial \Omega }\beta (x)\left\vert u(x)\right\vert
^{p(x)}d\sigma $ with $\beta ^{-}>0$. For any $u,u_{k}\in W_{\omega
_{1},\omega _{2}}^{1,p(.)}\left( \Omega \right) $ $\left( k=1,2,...\right) $%
, we have

\begin{enumerate}
\item[\textit{(i)}] $\left\Vert u\right\Vert _{\beta (x)}^{p^{-}}\leq
I_{\beta (x)}(u)\leq \left\Vert u\right\Vert _{\beta (x)}^{p^{+}}$ with $%
\left\Vert u\right\Vert _{\beta (x)}\geq 1,$

\item[\textit{(ii)}] $\left\Vert u\right\Vert _{\beta (x)}^{p^{+}}\leq
I_{\beta (x)}(u)\leq \left\Vert u\right\Vert _{\beta (x)}^{p^{-}}$ with $%
\left\Vert u\right\Vert _{\beta (x)}\leq 1,$

\item[\textit{(iii)}] $\min \left\{ \left\Vert u\right\Vert _{\beta
(x)}^{p^{-}},\left\Vert u\right\Vert _{\beta (x)}^{p^{+}}\right\} \leq
I_{\beta (x)}(u)\leq \max \left\{ \left\Vert u\right\Vert _{\beta
(x)}^{p^{-}},\left\Vert u\right\Vert _{\beta (x)}^{p^{+}}\right\} ,$

\item[\textit{(iv)}] $\left\Vert u-u_{k}\right\Vert _{\beta (x)}\rightarrow
0 $ if and only if $I_{\beta (x)}(u-u_{k})\rightarrow 0$ (as $k\rightarrow
\infty $),

\item[\textit{(v)}] $\left\Vert u_{k}\right\Vert _{\beta (x)}\rightarrow
\infty $ if and only if $I_{\beta (x)}(u_{k})\rightarrow \infty $ (as $%
k\rightarrow \infty $).
\end{enumerate}
\end{proposition}

The following Proposition can be proved by Proposition 2.2 in \cite{be}.

\begin{proposition}
\label{pro3}Let us define the functional $L_{\beta (x)}:W_{\omega
_{1},\omega _{2}}^{1,p(.)}\left( \Omega \right) \rightarrow 
%TCIMACRO{\U{211d} }%
%BeginExpansion
\mathbb{R}
%EndExpansion
$ by%
\begin{equation*}
L_{\beta (x)}\left( u\right) =\dint\limits_{\Omega }\frac{\omega _{1}(x)}{%
p\left( x\right) }\left\vert \nabla u\right\vert ^{p\left( x\right)
}dx+\int\limits_{\partial \Omega }\frac{\beta (x)}{p(x)}\left\vert
u(x)\right\vert ^{p(x)}d\sigma
\end{equation*}%
for all $u\in W_{\omega _{1},\omega _{2}}^{1,p(.)}\left( \Omega \right) $.
Then we obtain $L_{\beta (x)}\in C^{1}\left( W_{\omega _{1},\omega
_{2}}^{1,p(.)}\left( \Omega \right) ,%
%TCIMACRO{\U{211d} }%
%BeginExpansion
\mathbb{R}
%EndExpansion
\right) $ and 
\begin{equation*}
L_{\beta (x)}^{\prime }\left( u\right) (v)=<L_{\beta (x)}^{\prime }\left(
u\right) ,v>=\dint\limits_{\Omega }\omega _{1}(x)\left\vert \nabla
u\right\vert ^{p\left( x\right) -2}\nabla u\nabla vdx+\int\limits_{\partial
\Omega }\beta (x)\left\vert u(x)\right\vert ^{p(x)-2}uvd\sigma
\end{equation*}%
for any $u,v\in W_{\omega _{1},\omega _{2}}^{1,p(.)}\left( \Omega \right) $.
In addition, we have the following properties

\begin{enumerate}
\item[\textit{(i)}] $L_{\beta (x)}^{\prime }:W_{\omega _{1},\omega
_{2}}^{1,p(.)}\left( \Omega \right) \longrightarrow W_{\omega _{1}^{\ast
},\omega _{2}^{\ast }}^{-1,p^{\prime }(.)}\left( \Omega \right) $ is
continuous, bounded and strictly monotone operator,

\item[\textit{(ii)}] $L_{\beta (x)}^{\prime }:W_{\omega _{1},\omega
_{2}}^{1,p(.)}\left( \Omega \right) \longrightarrow W_{\omega _{1}^{\ast
},\omega _{2}^{\ast }}^{-1,p^{\prime }(.)}\left( \Omega \right) $ is a
mapping of type $\left( S_{+}\right) ,$ i.e., if $u_{n}\rightharpoonup u$ in 
$W_{\omega _{1},\omega _{2}}^{1,p(.)}\left( \Omega \right) $ and $%
\limsup_{n\longrightarrow \infty }L_{\beta (x)}^{\prime }\left( u_{n}\right)
(u_{n}-u)\leq 0$, then $u_{n}\longrightarrow u$ in $W_{\omega _{1},\omega
_{2}}^{1,p(.)}\left( \Omega \right) $

\item[\textit{(iii)}] $L_{\beta (x)}^{\prime }:W_{\omega _{1},\omega
_{2}}^{1,p(.)}\left( \Omega \right) \longrightarrow W_{\omega _{1}^{\ast
},\omega _{2}^{\ast }}^{-1,p^{\prime }(.)}\left( \Omega \right) $ is a
homeomorphism.
\end{enumerate}
\end{proposition}

\begin{theorem}
\label{teo5}(see \cite{ri}) Let $X$ be a separable and reflexive real Banach
space; $\Phi :X\rightarrow 
%TCIMACRO{\U{211d} }%
%BeginExpansion
\mathbb{R}
%EndExpansion
$ a continuously G\^{a}teaux differentiable and sequentially weakly lower
semicontinuous functional whose G\^{a}teaux derivative admits a continuous
inverse on $X^{\ast }$; $\Psi :X\rightarrow 
%TCIMACRO{\U{211d} }%
%BeginExpansion
\mathbb{R}
%EndExpansion
$ a continuously G\^{a}teaux differentiable functional whose G\^{a}teaux
derivative is compact. Moreover, assume that

\begin{enumerate}
\item[\textit{(i)}] $\underset{\left\Vert u\right\Vert \rightarrow \infty }{%
\lim }\left( \Phi (u)+\lambda \Psi (u)\right) =\infty $ for all $\lambda >0,$

\item[\textit{(ii)}] there are $r\in 
%TCIMACRO{\U{211d} }%
%BeginExpansion
\mathbb{R}
%EndExpansion
$ and $u_{0},u_{1}\in X$ such that $\Phi (u_{0})<r<\Phi (u_{1}),$

\item[\textit{(iii)}] $\underset{u\in \Phi ^{-1}\left( \left( -\infty ,r%
\right] \right) }{\inf }\Psi (u)>\frac{\left( \Phi (u_{1})-r\right) \Psi
(u_{0})+\left( r-\Phi (u_{0})\right) \Psi (u_{1})}{\Phi (u_{1})-\Phi (u_{0})}%
.$
\end{enumerate}

Then there exist an open interval $\Lambda \subset \left( 0,\infty \right) $
and a positive constant $\rho >0$ such that for any $\lambda \in \Lambda $
the equation $\Phi ^{\prime }\left( u\right) +\lambda \Psi ^{\prime }\left(
u\right) =0$ has at least three solutions in $X$ whose norms are less than $%
\rho $.
\end{theorem}

\section{The Main Result}

Throughout the paper we assume that the following conditions:

\begin{enumerate}
\item[\textit{(I)}] $\left\vert f(x,t)\right\vert \leq h(x)+c_{2}\left\vert
t\right\vert ^{s(x)-1}$ for any $(x,t)\in \Omega \times 
%TCIMACRO{\U{211d} }%
%BeginExpansion
\mathbb{R}
%EndExpansion
,$ $c_{2}>0$, where the function $f:$ $\Omega \times 
%TCIMACRO{\U{211d} }%
%BeginExpansion
\mathbb{R}
%EndExpansion
\rightarrow 
%TCIMACRO{\U{211d} }%
%BeginExpansion
\mathbb{R}
%EndExpansion
$ is a Carath\'{e}odory function, $h(x)\in L^{\frac{s(x)}{s(x)-1}}(\Omega )$%
, $h(x)\geq 0$ and $s(x)\in C_{+}\left( \Omega \right) $, $1<s^{-}=\underset{%
x\in \overline{\Omega }}{\inf }s(x)\leq s^{+}=\underset{x\in \overline{%
\Omega }}{\sup }s(x)<p^{-}$ with $s(x)<\left( p_{\ast }\right) ^{\ast }(x)$
for all $x\in \overline{\Omega }.$

\item[\textit{(II)}] \textit{(i) }$f(x,t)<0$ for all $(x,t)\in \Omega \times 
%TCIMACRO{\U{211d} }%
%BeginExpansion
\mathbb{R}
%EndExpansion
,$ and $\left\vert t\right\vert \in \left( 0,1\right) $,

\textit{(ii) }$f(x,t)\geq k>0,$ when $\left\vert t\right\vert \in \left(
t_{0},\infty \right) $, $t_{0}>1$.
\end{enumerate}

Let $u\in W_{\omega _{1},\omega _{2}}^{1,p(.)}\left( \Omega \right) .$ Then
the functional $\Phi _{\lambda }\left( u\right) $ is defined by 
\begin{equation*}
\Phi _{\lambda }\left( u\right) =L_{\beta (x)}(u)+\lambda \Psi (u),
\end{equation*}

where $\Psi (u)=-\int\limits_{\Omega }F(x,u)dx$ and $F(x,t)=\int%
\limits_{0}^{t}f(x,y)dy.$ Moreover, $\Phi _{\lambda }\left( u\right) $ is
called energy functional of the problem (\ref{P}).

It is obvious that $\left( L_{\beta (x)}^{\prime }\right) ^{-1}:W_{\omega
_{1}^{\ast },\omega _{2}^{\ast }}^{-1,q(.)}\left( \Omega \right)
\longrightarrow W_{\omega _{1},\omega _{2}}^{1,p(.)}\left( \Omega \right) $
exists and continuous, because $L_{\beta (x)}^{\prime }:W_{\omega
_{1},\omega _{2}}^{1,p(.)}\left( \Omega \right) \longrightarrow W_{\omega
_{1}^{\ast },\omega _{2}^{\ast }}^{-1,q(.)}\left( \Omega \right) $ is a
homeomorphism by Proposition \ref{pro3}. Moreover, due to the assumption $%
(I) $ it is well known that $\Psi \in C^{1}\left( W_{\omega _{1},\omega
_{2}}^{1,p(.)}\left( \Omega \right) ,%
%TCIMACRO{\U{211d} }%
%BeginExpansion
\mathbb{R}
%EndExpansion
\right) $ with the derivatives given by $\left\langle \Psi ^{\prime }\left(
u\right) ,v\right\rangle =-\int\limits_{\Omega }f(x,u)vdx$ for any $u,v\in
W_{\omega _{1},\omega _{2}}^{1,p(.)}\left( \Omega \right) $, and $\Psi
^{\prime }:W_{\omega _{1},\omega _{2}}^{1,p(.)}\left( \Omega \right)
\longrightarrow W_{\omega _{1}^{\ast },\omega _{2}^{\ast }}^{-1,q(.)}\left(
\Omega \right) $ is completely continuous by \cite[Theorem 2.9]{ala}.
Therefore, $\Psi ^{\prime }:W_{\omega _{1},\omega _{2}}^{1,p(.)}\left(
\Omega \right) \longrightarrow W_{\omega _{1}^{\ast },\omega _{2}^{\ast
}}^{-1,q(.)}\left( \Omega \right) $ is compact.

\begin{definition}
We call that $u\in W_{\omega _{1},\omega _{2}}^{1,p(.)}\left( \Omega \right) 
$ is a weak solution of the problem (\ref{P}) if%
\begin{equation*}
\dint\limits_{\Omega }\omega _{1}(x)\left\vert \nabla u\right\vert ^{p\left(
x\right) -2}\nabla u\nabla vdx+\int\limits_{\partial \Omega }\beta
(x)\left\vert u(x)\right\vert ^{p(x)-2}uvd\sigma -\lambda
\tint\limits_{\Omega }\omega _{2}(x)f(x,u)vdx=0
\end{equation*}%
for all $v\in W_{\omega _{1},\omega _{2}}^{1,p(.)}\left( \Omega \right) .$
We point out that if $\lambda \in 
%TCIMACRO{\U{211d} }%
%BeginExpansion
\mathbb{R}
%EndExpansion
$ is an eigenvalue of the problem (\ref{P}), then the corresponding $u\in
W_{\omega _{1},\omega _{2}}^{1,p(.)}\left( \Omega \right) -\left\{ 0\right\} 
$ is a weak solution of (\ref{P}).
\end{definition}

\begin{theorem}
There exist an open interval $\Lambda \subset \left( 0,\infty \right) $ and
a positive constant $\rho >0$ such that for any $\lambda \in \Lambda $, the
problem (\ref{P}) has at least three solutions in $W_{\omega _{1},\omega
_{2}}^{1,p(.)}\left( \Omega \right) $ whose norms are less than $\rho $.
\end{theorem}

\begin{proof}
We only need to prove the conditions \textit{(i)}, \textit{(ii)} and \textit{%
(iii)} in Theorem \ref{teo5}. Using Proposition \ref{pro3} we get%
\begin{eqnarray}
L_{\beta (x)}\left( u\right) &=&\dint\limits_{\Omega }\frac{\omega _{1}(x)}{%
p\left( x\right) }\left\vert \nabla u\right\vert ^{p\left( x\right)
}dx+\int\limits_{\partial \Omega }\frac{\beta (x)}{p(x)}\left\vert
u(x)\right\vert ^{p(x)}d\sigma  \notag \\
&=&\frac{1}{p^{+}}I_{\beta (x)}(u)  \notag \\
&\geq &\frac{1}{p^{+}}\left\Vert u\right\Vert _{\beta (x)}^{p^{-}}
\label{3.1}
\end{eqnarray}%
for any $u\in W_{\omega _{1},\omega _{2}}^{1,p(.)}\left( \Omega \right) $
with $\left\Vert u\right\Vert _{\beta (x)}>1$.

In addition, due to $(I)$ and H\"{o}lder inequality, we have 
\begin{eqnarray}
-\Psi (u) &=&\int\limits_{\Omega }F(x,u)dx=\int\limits_{\Omega }\left(
\int\limits_{0}^{u(x)}f(x,t)dt\right) dx  \notag \\
&\leq &\int\limits_{\Omega }\left( h(x)\left\vert u(x)\right\vert +\frac{%
c_{2}}{s(x)}\left\vert u(x)\right\vert ^{s(x)}\right) dx  \notag \\
&\leq &2\left\Vert h\right\Vert _{\frac{s(.)}{s(.)-1},\Omega }\left\Vert
u\right\Vert _{s(.),\Omega }+\frac{c_{2}}{s^{-}}\int\limits_{\Omega
}\left\vert u(x)\right\vert ^{s(x)}dx.  \label{3.2}
\end{eqnarray}%
By Corollary \ref{cor4}, there exist the continuous embedding $W_{\omega
_{1},\omega _{2}}^{1,p(.)}\left( \Omega \right) \hookrightarrow
L^{s(.)}(\Omega )$ and the inequality 
\begin{equation}
\int\limits_{\Omega }\left\vert u(x)\right\vert ^{s(x)}dx\leq \max \left\{
\left\Vert u\right\Vert _{s(.),\Omega }^{s^{-}},\left\Vert u\right\Vert
_{s(.),\Omega }^{s^{+}}\right\} \leq c_{3}\left\Vert u\right\Vert _{\beta
(x)}^{s^{+}}\text{.}  \label{3.3}
\end{equation}%
If we use (\ref{3.2}) and (\ref{3.3}), then we get%
\begin{equation}
-\Psi (u)\leq 2C_{7}\left\Vert h\right\Vert _{\frac{s(.)}{s(.)-1},\Omega
}\left\Vert u\right\Vert _{\beta (x)}+\frac{c_{4}}{s^{-}}\left\Vert
u\right\Vert _{\beta (x)}^{s^{+}}.  \label{3.4}
\end{equation}%
For any $\lambda >0$ we can write%
\begin{equation*}
L_{\beta (x)}\left( u\right) +\lambda \Psi (u)\geq \frac{1}{p^{+}}\left\Vert
u\right\Vert _{\beta (x)}^{p^{-}}-2\lambda C_{7}\left\Vert h\right\Vert _{%
\frac{s(.)}{s(.)-1},\Omega }\left\Vert u\right\Vert _{\beta (x)}-\frac{c_{4}%
}{s^{-}}\lambda \left\Vert u\right\Vert _{\beta (x)}^{s^{+}}
\end{equation*}%
by (\ref{3.1}) and (\ref{3.4}). Since $1<s^{+}<p^{-}$, $\underset{\left\Vert
u\right\Vert _{\beta (x)}\rightarrow \infty }{\lim }\left( L_{\beta
(x)}\left( u\right) +\lambda \Psi (u)\right) =\infty $ for all $\lambda >0$
and the proof of \textit{(i)} is completed.

Due to $\frac{\partial F(x,t)}{\partial t}=f(x,t)$ and $(II)$, it is easy to
see that $F(x,t)$ is increasing and decreasing for $t\in \left( t_{0},\infty
\right) $ and $\left( 0,1\right) $ with respect to $x\in \Omega ,$
respectively. Since $F(x,t)\geq kt$ uniformly for $x$, we have $%
F(x,t)\rightarrow \infty $ as $t\rightarrow \infty $. Then for a real number 
$\delta >t_{0},$ we can obtain 
\begin{equation}
F(x,t)\geq 0=F(x,0)\geq F(x,\tau ),\text{ for all }x\in \Omega ,\text{ }%
t>\delta ,\text{ }\tau \in \left( 0,1\right) .  \label{3.5}
\end{equation}%
Let $\beta ,\gamma $ be two real numbers such that $0<\beta <\min \left\{
1,c_{1}\right\} $, where $c_{1}$ is given in Corollary \ref{cor1}, and $%
\gamma >\delta $ $(\gamma >1)$ satisfies $\gamma ^{p^{-}}\left\Vert \beta
\right\Vert _{1,\partial \Omega }>1$. If we use relation (\ref{3.5}), then
we have $F(x,t)\leq F(x,0)=0$ for $t\in \left[ 0,\beta \right] $, and%
\begin{equation}
\int\limits_{\Omega }\sup_{0\leq t\leq \beta }F(x,t)dx\leq
\int\limits_{\Omega }F(x,0)dx=0.  \label{3.6}
\end{equation}%
Using $\gamma >\delta $ and (\ref{3.5}), we have $\int\limits_{\Omega
}F(x,\delta )dx>0$ and 
\begin{equation}
\frac{1}{c_{1}^{p^{+}}}\frac{\beta ^{+}}{\gamma ^{p^{-}}}\int\limits_{\Omega
}F(x,\delta )dx>0.  \label{3.7}
\end{equation}%
If we use the inequalities in (\ref{3.6}) and (\ref{3.7}), then we get%
\begin{equation*}
\int\limits_{\Omega }\sup_{0\leq t\leq a}F(x,t)dx\leq 0<\frac{1}{%
c_{1}^{p^{+}}}\frac{\beta ^{+}}{\gamma ^{p^{-}}}\int\limits_{\Omega
}F(x,\delta )dx.
\end{equation*}%
Define $u_{0},u_{1}\in W_{\omega _{1},\omega _{2}}^{1,p(.)}\left( \Omega
\right) $ with $u_{0}(x)=0$ and $u_{1}(x)=\gamma $ for any $x\in \Omega .$
If we take $r=\frac{1}{p^{+}}\left( \frac{\beta }{c_{1}}\right) ^{p^{+}}$,
then $r\in \left( 0,1\right) ,$ $L_{\beta (x)}(u_{0})=\Psi (u_{0})=0$ and 
\begin{eqnarray*}
L_{\beta (x)}(u_{1}) &=&\int\limits_{\partial \Omega }\frac{\beta (x)}{p(x)}%
\gamma ^{p(x)}d\sigma \geq \frac{\gamma ^{p^{-}}}{p^{+}}\int\limits_{%
\partial \Omega }\beta (x)d\sigma =\frac{1}{p^{+}}\gamma ^{p^{-}}\left\Vert
\beta \right\Vert _{1,\partial \Omega } \\
&\geq &\frac{1}{p^{+}}>r.
\end{eqnarray*}%
Thus we have $L_{\beta (x)}(u_{0})<r<L_{\beta (x)}(u_{1})$ and%
\begin{equation*}
\Psi (u_{1})=-\int\limits_{\Omega }F(x,u_{1})dx=-\int\limits_{\Omega
}F(x,\gamma )dx<0.
\end{equation*}%
Then the proof of \textit{(ii)} is obtained.

On the other hand, we have%
\begin{eqnarray*}
-\frac{\left( L_{\beta (x)}(u_{1})-r\right) \Psi (u_{0})+\left( r-L_{\beta
(x)}(u_{0})\right) \Psi (u_{1})}{L_{\beta (x)}(u_{1})-L_{\beta (x)}(u_{0})}
&=&-r\frac{\Psi (u_{1})}{L_{\beta (x)}(u_{1})} \\
&=&r\frac{\int\limits_{\Omega }F(x,\gamma )dx}{\int\limits_{\partial \Omega }%
\frac{\beta (x)}{p(x)}\gamma ^{p(x)}d\sigma }>0.
\end{eqnarray*}%
Now, let $u\in W_{\omega _{1},\omega _{2}}^{1,p(.)}\left( \Omega \right) $
with $L_{\beta (x)}(u)\leq r<1.$ Since $\frac{1}{p^{+}}I_{\beta (x)}(u)\leq
L_{\beta (x)}(u)\leq r$, we obtain 
\begin{equation*}
I_{\beta (x)}(u)\leq p^{+}r=\left( \frac{\beta }{c_{1}}\right) ^{p^{+}}<1.
\end{equation*}%
Due to Proposition \ref{pro2}, we see that $\left\Vert u\right\Vert _{\beta
(x)}<1$ and 
\begin{equation*}
\frac{1}{p^{+}}\left\Vert u\right\Vert _{\beta (x)}^{p^{+}}\leq \frac{1}{%
p^{+}}I_{\beta (x)}(u)\leq L_{\beta (x)}(u)\leq r.
\end{equation*}%
Then using Corollary \ref{cor1}, we can get 
\begin{equation*}
\left\vert u(x)\right\vert \leq c_{1}\left\Vert u\right\Vert _{\beta
(x)}\leq c_{1}\left( p^{+}r\right) ^{\frac{1}{p^{+}}}=\beta
\end{equation*}%
for all $u\in W_{\omega _{1},\omega _{2}}^{1,p(.)}\left( \Omega \right) $
and $x\in \Omega $ with $\Phi (u)\leq r.$

The last inequality implies that%
\begin{equation*}
-\underset{u\in \Phi ^{-1}\left( \left( -\infty ,r\right] \right) }{\inf }%
\Psi (u)=\underset{u\in \Phi ^{-1}\left( \left( -\infty ,r\right] \right) }{%
\sup }-\Psi (u)\leq \int\limits_{\Omega }\sup_{0\leq t\leq \beta
}F(x,t)dx\leq 0.
\end{equation*}%
Then we have%
\begin{equation*}
-\underset{u\in \Phi ^{-1}\left( \left( -\infty ,r\right] \right) }{\inf }%
\Psi (u)<r\frac{\int\limits_{\Omega }F(x,\gamma )dx}{\int\limits_{\partial
\Omega }\frac{\beta (x)}{p(x)}\gamma ^{p(x)}d\sigma }
\end{equation*}%
and%
\begin{equation*}
\underset{u\in \Phi ^{-1}\left( \left( -\infty ,r\right] \right) }{\inf }%
\Psi (u)>\frac{\left( \Phi (u_{1})-r\right) \Psi (u_{0})+\left( r-\Phi
(u_{0})\right) \Psi (u_{1})}{\Phi (u_{1})-\Phi (u_{0})}.
\end{equation*}%
This completes the proof.
\end{proof}


\begin{thebibliography}{99}
\bibitem{al} M. Allaoui, Robin problems involving the $p(x)$-Laplacian.
Appl. Math. and Comp. 332 (2018), 457-468.

\bibitem{ala} M. Allaoui, A. R. El Amrouss, A. Ourraoui, Existence and
multiplicity of solutions for a Steklov problem involving the $p(x)$-Laplace
operator. Electron. J. Differential Equations 2012(132) (2012), 1-12.

\bibitem{ayd} I. Aydin, Weighted variable Sobolev spaces and capacity. J.
Funct. Spaces Appl. 2012, (2012).

\bibitem{ay} I. Aydin, C. Unal, Weighted stochastic field exponent Sobolev
spaces and nonlinear degenerated elliptic problem with nonstandard growth.
Hacettepe J. of Math. and Stat. 49(4) (2020), 1383-1396.

\bibitem{ayn} I. Aydin, C. Unal, The Kolmogorov--Riesz theorem and some
compactness criterions of bounded subsets in weighted variable exponent
amalgam and Sobolev spaces. Collect. Math. 71 (2020), 349-367.

\bibitem{aydn} I. Aydin, C. Unal, Three solutions to a Steklov problem
involving the weighted $p(.)$-Laplacian. Rocky Mountain J. Math. (Accepted
for publication) (2020), https://arxiv.org/pdf/2005.10344.pdf.

\bibitem{Ch} N. T. Chung, Some remarks on a class of $p(x)-$Laplacian Robin
eigenvalue problems. Mediterr. J. Math. 15(147) (2018), 1-14.

\bibitem{de} S. G. Deng, Eigenvalues of the $p(x)$-Laplacian Steklov
problem. J. Math. Anal. Appl. 339 (2008), 925-937.

\bibitem{den} S. G. Deng, A local mountain pass theorem and applications to
a double perturbed $p(x)$-Laplacian equations. Appl. Math. Comput. 211
(2009), 234-241.

\bibitem{deng} S. G. Deng, Positive solutions for Robin problem involving
the $p(x)$-Laplacian. J. Math. Anal. Appl. 360 (2009), 548-560.

\bibitem{fan} X. L. Fan, Solutions for $p(x)$-Laplacian Dirichlet problems
with singular coefficients. J. Math. Anal. Appl. 312 (2005), 464-477.

\bibitem{FanZ} X. L. Fan, Q. Zhang, Existence of solutions for $p\left(
x\right) $-Laplacian Dirichlet problem. Nonlinear Anal. 52 (2003), 1843-1852.

\bibitem{be} B. Ge, Q. M. Zhou, Multiple solutions for a Robin-type
differential inclusion problem involving the $p(x)$-Laplacian. Math. Methods
Appl. Sci. 40(18) (2017), 6229-6238.

\bibitem{Ha} T. C. Halsey, Electrorheological fluids, Science 258(5083)
(1992). 761-766.

\bibitem{ke} K. Kefi, On the Robin problem with indefinite weight in Sobolev
spaces with variable exponents. Z. Anal. Anwend. 37 (2018), 25-38.

\bibitem{ko} O. Kov\'{a}\v{c}ik, J. R\'{a}kosn\'{\i}k, On spaces $L^{p(x)}$
and $W^{k,p(x)}$. Czechoslovak Math. J. 41(116)(4) (1991), 592-618.

\bibitem{mi} M. Mih\u{a}ilescu, Existence and multiplicity of solutions for
a Neumann problem involving the $p(x)$-Laplace operator. Nonlinear Anal. 67
(2007), 1419-1425.

\bibitem{mih} M. Mih\u{a}ilescu, V. R\u{a}dulescu, A multiplicity result for
a nonlinear degenerate problem arising in the theory of electrorheological
fluids. Proc. R. Soc. Lond. Ser. A 462 (2006), 2625-2641.

\bibitem{ri} B. Ricceri, On three critical points theorem. Arch. Math.
(Basel) 75 (2000), 220-226.

\bibitem{Ru} M. R\r{u}\v{z}i\v{c}ka, Electrorheological Fluids: Modeling and
Mathematical Theory. Lecture Notes in Mathematics 1748, Springer, 2000.

\bibitem{Tso} N. Tsouli, O. Darhouche, Existence and multiplicity results
for nonlinear problems involving the $p(x)$-Laplace operator. Opuscula Math.
34(3) (2014), 621-638.

\bibitem{Ts} N. Tsouli, O. Chakrone, O. Darhouche, M. Rahmani, Existence and
multiplicity of solutions for a Robin problem involving the $p(x)$-Laplace
operator. Hindawi Publishing Corporation, Conference Papers in Mathematics
Volume 2013, Article ID 231898, 7 pp.

\bibitem{UnalAy2} C. Unal, I. Aydin, Compact embeddings of weighted variable
exponent Sobolev spaces and existence of solutions for weighted $p(.)$%
-Laplacian. Complex Var. Elliptic Equ. (Accepted for publication) (2020).

\bibitem{Zh} V. V. Zhikov, Averaging of functionals of tha calculus of
variations and elasticity theory. Math. USSR Izv. 9 (1987), 33-66.
\end{thebibliography}
\end{document}